\documentclass[reqno,11pt]{amsart}
\usepackage{epsfig,amscd,amssymb,amsmath,amsfonts}
\usepackage{lscape}
\usepackage{amsmath}
\usepackage{graphicx}
\usepackage{amsthm,color}
\usepackage{tikz}
\usepackage{float}
\usepackage{xparse}

\usetikzlibrary{shapes.geometric}
\usetikzlibrary{shapes,positioning,intersections,quotes}
\usetikzlibrary{graphs}
\usetikzlibrary{graphs,quotes}
\usetikzlibrary{decorations.pathmorphing}

\tikzset{snake it/.style={decorate, decoration=snake}}
\tikzset{snake it/.style={decorate, decoration=snake}}

\usetikzlibrary{decorations.pathreplacing,decorations.markings,snakes}

\newtheorem{theorem}{Theorem}[section]
\newtheorem{lemma}[theorem]{Lemma}
\newtheorem{proposition}{Proposition}[section]
\theoremstyle{definition}
\newtheorem{definition}[theorem]{Definition}
\newtheorem{corollary}[theorem]{Corollary}

\newtheorem{example}[theorem]{Example}

\theoremstyle{remark}
\newtheorem{remark}[theorem]{Remark}

\numberwithin{equation}{section}

\usepackage[margin=1.05in]{geometry}
\usepackage[colorlinks]{hyperref}
\usepackage{siunitx}

\begin{document}

\title[ Existence of the $det^{S^2}$ map]{ Existence of the $det^{S^2}$ map}

\author{Mihai D. Staic}
\address{Department of Mathematics and Statistics, Bowling Green State University, Bowling Green, OH 43403 }
\address{Institute of Mathematics of the Romanian Academy, PO.BOX 1-764, RO-70700 Bu\-cha\-rest, Romania.}

\email{mstaic@bgsu.edu}






\begin{abstract} In this paper we show that  for a vector space $V_d$ of dimension $d$ there exists a linear map $det^{S^2}:V_d^{\otimes d(2d-1)}\to k$ with  the property that $det^{S^2}(\otimes_{1\leq i<j\leq 2d}(v_{i,j}))=0$ if there exists $1\leq x<y<z\leq 2d$ such that $v_{x,y}=v_{x,z}=v_{y,z}$. The existence of such a map was conjectured in  \cite{sta2}. 
We present two applications of the map $det^{S^2}$ to geometry and combinatorics. 
\end{abstract}
\maketitle



%

\section{Introduction}

The determinant of a matrix plays an important role in several areas of mathematics. It captures quantitative information (like area of a region, volume of a solid), but also  qualitative  information (like linear dependence of $d$ vectors in a $d$-dimensional vector space). Heuristically, the best way to introduce the determinant of a linear transformation $T:V_d\to V_d$, is to consider the exterior algebra of the $d$-dimensional vector space $V_d$, and then define the determinant as the unique constant that determines the map $\Lambda(T):\Lambda_{V_d}[d]\to \Lambda_{V_d}[d]$. Equivalently, one can show that the determinant is the unique (up to a scalar) nontrivial linear map $det:V_d^{\otimes d}\to k$ with the property that $det(\otimes_{1\leq i\leq d} (v_i))=0$ if there exist $1\leq x<y\leq d$ such that $v_x=v_y$.

The graded vector space $\Lambda_{V_d}^{S^2}$ was introduced in \cite{sta2} as a generalization of the exterior algebra. It has properties similar with the ones of the exterior algebra, for example $\Lambda_{V_d}^{S^2}[n]=0$ if $n>2d$. It was conjectured in \cite{sta2} that $dim_k(\Lambda_{V_d}^{S^2}[2d])=1$. This conjecture is equivalent with the existence and uniqueness (up to a scalar) of a nontrivial linear map $det^{S^2}:V_d^{\otimes d(2d-1)}\to k$  with the property that $det(\otimes_{1\leq i<j\leq 2d} (v_{i,j}))=0$ if there exist $1\leq x<y<z\leq 2d$ such that $v_{x,y}=v_{x,z}=v_{y,z}$. The conjecture was checked to be true in the case $d=2$ (\cite{sta2}) and $d=3$ (\cite{lss}). 

In this paper we show for every $d$ there exists a nontrivial map $det^{S^2}:V_d^{\otimes d(2d-1)}\to k$ with the above mentioned property. For this we consider a system of $2d$ vector equations associate to $(v_{i,j})_{1\leq i<j\leq 2d}\in V_d^{d(2d-1)}$. The corresponding matrix is of dimension $2d^2\times d(2d-1)$, but we can eliminate one of the vector equations to get a square $d(2d-1)\times d(2d-1)$ matrix. The determinant of this square matrix is nontrivial, and has the universality property we are looking for. 
In particular  this shows that $dim_k(\Lambda_{V_d}^{S^2}[2d])\geq 1$. The uniqueness of the map $det^{S^2}$  is still an open question. 

As an application we give a geometrical interpretation of the condition $det^{S^2}(\otimes_{1\leq i<j\leq 2d} (v_{i,j}))=0$. In particular, we show that if $p_i\in V_d$  for $1\leq i\leq 2d$, and we take $v_{i,j}=p_j-p_i$ for all $1\leq i<j\leq 2d$ then $det^{S^2}(\otimes_{1\leq i<j\leq 2d} (v_{i,j}))=0$. This a generalization of a result proved in \cite{sv} for the  case $d=2$ and $d=3$. 

We also give an application to combinatorics. 
More precisely we show that $(\Gamma_1,\dots,\Gamma_d)$ is cycle free $d$-partition of the complete graph $K_{2d}$   if and only if $det^{S^2}(f_{(\Gamma_1,\dots,\Gamma_d)})\neq 0$ (where $f_{(\Gamma_1,\dots,\Gamma_d)}\in V_d^{d(2d-1)}$ is a certain  element associated to  $(\Gamma_1,\dots,\Gamma_d)$). The case $d=2$ and $d=3$ was proved in \cite{lss}.   One can think about this result as a generalization of the fact that a $d\times d$ matrix that has $d$ entries equal to $1$ and the rest of the entries equal to zero will have a nonzero determinant if and only if it has a nonzero entry in every row and every column.

\section{Preliminary}

In this paper $k$ is a field, $V_d$ is a $d$-dimensional vector space, and $\mathcal{B}_d=\{e_1,\dots,e_d\}$ is a fixed basis for $V_d$. We denote by $V_d^{\otimes m}$ the $m$-th tensor power of $V_d$. 

The exterior algebra $\Lambda_{V_d}$ can be defined as the quotient of the tensor algebra $\displaystyle{T_{V_d}=\oplus_n V_d^{\otimes n}}$ by the ideal generated by elements of the form $u\otimes u$ where $u\in V_d$. It is well know that $dim_k(\Lambda_{V_d}[d])=1$, in particular if $T:V_d\to V_d$ is a linear map then $\Lambda(T):\Lambda_{V_d}[d]\to \Lambda_{V_d}[d]$ is the  multiplication by a constant, which by definition is denoted by $det(T)$. Alternatively, one can define the determinant as the unique nontrivial linear map $det:V_d^{\otimes d}\to k$ with the property that $det(\otimes_{1\leq i\leq d}(v_i))=0$ if there exist $1\leq x<y\leq d$ such that $v_{x}=v_{y}$. 

Next we recall from \cite{lss}, \cite{sta2}, and  \cite{sv} a few results about $\Lambda_{V_d}^{S^2}$ and the $det^{S^2}$ map. For every $n\geq 0$ we define 
$$\Lambda_{V_d}^{S^2}[n]=\frac{\mathcal{T}^{S^2}_{V_d}[n]}{\mathcal{E}^{S^2}_{V_d}[n]},$$
where $\mathcal{T}^{S^2}_{V_d}[n]=V_d^{\otimes \frac{n(n-1)}{2}}$, and   $\mathcal{E}^{S^2}_{V_d}[n]$ is the subspace of $\mathcal{T}^{S^2}_{V_d}[n]$ generated by those elements $\otimes_{1\leq i<j\leq n}(v_{i,j})\in V_d^{\otimes \frac{n(n-1)}{2}}$ with the property that there exists $1\leq x<y<z\leq n $ such that $v_{x,y}=v_{x,z}=v_{y,z}$. Notice that we use a slightly different notation from the one in \cite{sta2}, more precisely the grading of $\Lambda_{V_d}^{S^2}$ is shifted by $1$ (i.e. $\Lambda_{V_d}^{S^2}[n]=\Lambda_{V_d}^{S^2}(n+1)$). This is more consistent with the usual grading on the exterior algebra. 

It was shown in \cite{sta2} that $\Lambda_{V_d}^{S^2}[n]=0$ if $n>2d$. It was conjectured in the same paper that
$dim_k(\Lambda_{V_d}^{S^2}[2d])=1$. This conjecture is equivalent with the existence and uniqueness (up to a constant) of a nontrivial linear map $det^{S^2}:V_d^{\otimes d(2d-1)}\to k$ such that $det^{S^2}(\otimes_{1\leq i< j\leq 2d}(v_{i,j}))=0$ if there exist $1\leq x<y<z\leq 2d$ such that $v_{x,y}=v_{x,z}=v_{y,z}$. Notice the similitude with the determinant map. 

The conjecture was checked for $d=2$ in \cite{sta2}, and for $d=3$ in \cite{lss}. In particular for $d=2$ and $d=3$ there exists a map $det^{S^2}:V_d^{\otimes d(2d-1)}\to k$ with the above mentioned property.  In this paper we show $dim_k(\Lambda_{V_d}^{S^2}[2d])\geq 1$, i.e. we prove the existence of a nontrivial map $det^{S^2}$ for any $d$. The uniqueness is still an open question for $d>3$.

We denote by $E_d$ the element $\otimes_{1\leq i< j\leq 2d}(e_{i,j})\in V_d^{\otimes d(2d-1)}$ determined by 
\begin{equation}
e_{i,j}=\begin{cases}	
  	e_t ~~{\rm if}~~i<2t-1,~~i ~~{\rm is ~ odd}~~ {\rm and }~~j=2t-1,\\
		e_t ~~{\rm if}~~i<2t,~~i ~~{\rm is ~ even} ~~{\rm and }~~j=2t,\\
		e_t ~~{\rm if}~~i=2t-1, ~ j>2t-1~ {\rm and }~~j ~~{\rm is ~ even},\\
		e_t ~~{\rm if}~~i=2t,~ j>2t~{\rm and }~~j ~~{\rm is ~ odd}.\\
\end{cases}
\label{defEd}
\end{equation}
It was shown in \cite{lss} that if $d=2$ or $d=3$ then $det^{S^2}(E_d)=1$ (in particular this shows that the map $det^{S^2}$ is nontrivial in those two cases). 
\begin{remark}
When $d=2$ it was shown in \cite{sv} that $det(\otimes_{1\leq i<j\leq 4}(v_{i,j}))=0$ if and only if there exist $p_1,p_2,p_3,p_4\in V_2$, and $\lambda_{i,j}\in k$ not all trivial such that $\lambda_{i,j}v_{i,j}=p_j-p_i$.  A similar but partial result is also true when $d=3$; more precisely if $q_1,q_2,q_3,q_4,q_5,q_6\in V_3$ and we take $w_{i,j}=q_j-q_i$ then $det^{S^2}(\otimes_{1\leq i<j\leq 6}(w_{i,j}))=0$. Later in the paper we will generalize this result. 
\end{remark}

Next we recall from \cite{lss}  a few definitions and examples of $d$-partitions of the complete graph $K_{2d}$. 

\begin{definition}  A $d$-partition of the complete graph $K_{2d}$ is an ordered collection $(\Gamma_1,\Gamma_2,\dots,\Gamma_d)$ of sub-graphs $\Gamma_i$ of  $K_{2d}$ such that:\\
1) $V(\Gamma_i)=V(K_{2d})$ for all $1\leq i \leq d$,  \\
2) $E(\Gamma_i)\cap E(\Gamma_j)=\emptyset$ for all $i\neq j$, \\
3) $\displaystyle{\bigcup_{i=1}^nE(\Gamma_i)=E(K_{2d})}$. \\
We say that the $d$-partition $(\Gamma_1,\Gamma_2,\dots,\Gamma_d)$ in homogeneous if $|E(\Gamma_i)\vert=|E(\Gamma_j)\vert$ for all $1\leq i<j\leq d$. We say that the partition $(\Gamma_1,\Gamma_2,\dots,\Gamma_d)$  is cycle-free if each $\Gamma_i$ is cycle-free.
\end{definition}

Let $\mathcal{B}_d=\{e_1,e_2,\dots,e_d\}$ be a basis for the vector space $V_d$. It was noticed in \cite{lss}  that the set $$\mathcal{G}_{\mathcal{B}_d}[2d]=\{\otimes_{1\leq i<j\leq 2d}(v_{i,j})\in V_d^{\otimes d(2d-1)}\; \vert \; v_{i,j}\in \mathcal{B}_d\},$$ (which is a basis for $V_d^{\otimes d(2d-1)}$) is in bijection with the set $\mathcal{P}_d(K_{2d})$ of $d$-partitions of the complete graph $K_{2d}$.  
Indeed, if $f=\otimes_{1\leq i<j\leq 2d}(v_{i,j})\in\mathcal{G}_{\mathcal{B}_d}[2d]$ we consider the  sub-graphs $\Gamma_i(f)$ of $K_{2d}$ constructed as follows: for every $1\leq i\leq d$ we take $V(\Gamma_i(f))=\{1,2,\dots,2d\}$ and $E(\Gamma_i(f))=\{(s,t)|v_{s,t}=e_i\}$. One can easily see that $\Gamma(f)=(\Gamma_1(f),\dots,\Gamma_d(f))$ is a $d$-partition of $K_{2d}$.

Moreover, the map $f\mapsto \Gamma(f)$  is a bijection from the set $\mathcal{G}_{\mathcal{B}_d}[2d]$ to the set of $d$-partitions of $K_{2d}$. We will denote by $f_{(\Gamma_1,\dots,\Gamma_d)}$ the element in $\mathcal{G}_{\mathcal{B}_d}[2d]$ corresponding to the partition $(\Gamma_1,\dots,\Gamma_d)$. 

\begin{example} Let $$E_3=\begin{pmatrix}
	1& e_1&e_2&e_1&e_3&e_1\\
	&1&e_1&e_2&e_1&e_3\\
	& &1&e_2&e_3&e_2\\
&&&1&e_2&e_3\\
&&&&1&e_3\\
\otimes&&&&&1\\
\end{pmatrix}\in V_3^{\otimes 15}, $$
then the corresponding $3$-partition $\Gamma (E_3)$ of $K_6$ is given in  Figure \ref{E3}. Notice that $\Gamma (E_3)$ is homogeneous and cycle free. See \cite{lss} for more examples. 

\begin{figure}[h!]
	\centering
	\begin{tikzpicture}
		[scale=1.5,auto=left,every node/.style={shape = circle, draw, fill = white,minimum size = 1pt, inner sep=0.3pt}]
		\node (n1) at (0,0) {1};
		\node (n2) at (0.5,0.85)  {2};
		\node (n3) at (1.5,0.85)  {3};
		\node (n4) at (2,0)  {4};
		\node (n5) at (1.5,-0.85)  {5};
		\node (n6) at (0.5,-0.85)  {6};
		\foreach \from/\to in {n1/n2,n1/n4,n1/n6,n2/n3,n2/n5}
		\draw[line width=0.5mm,red]  (\from) -- (\to);	
		\node (n11) at (3,0) {1};
		\node (n21) at (3.5,0.85)  {2};
		\node (n31) at (4.5,0.85)  {3};
		\node (n41) at (5,0)  {4};
		\node (n51) at (4.5,-0.85)  {5};
		\node (n61) at (3.5,-0.85)  {6};
		\foreach \from/\to in {n11/n31,n21/n41,n31/n41,n31/n61,n41/n51}
		\draw[line width=0.5mm,orange]  (\from) -- (\to);	
		
		\node (n12) at (6,0) {1};
		\node (n22) at (6.5,0.85)  {2};
		\node (n32) at (7.5,0.85)  {3};
		\node (n42) at (8,0)  {4};
		\node (n52) at (7.5,-0.85)  {5};
		\node (n62) at (6.5,-0.85)  {6};
		\foreach \from/\to in {n12/n52,n22/n62,n32/n52,n42/n62,n52/n62}
		\draw[line width=0.5mm,blue]  (\from) -- (\to);	
		
	\end{tikzpicture}
	\caption{$\Gamma (E_3)=(\Gamma_1,\Gamma_2,\Gamma_3)$ the $3$-partition associated to $E_3$ \label{E3}}
\end{figure}
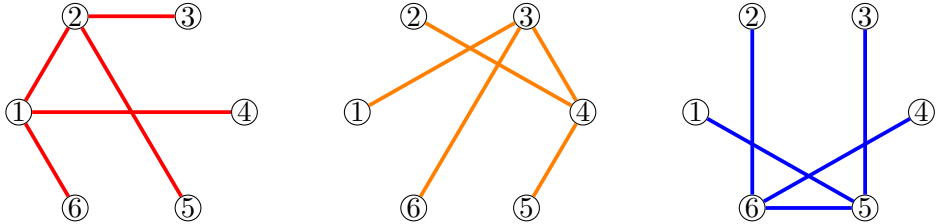
\end{example}

The following results were proved in \cite{lss}.

\begin{lemma}\label{generators} Take $(\Gamma_1,\dots, \Gamma_d)$ a $d$-partition of $K_{2d}$ that is not cycle-free. Then $\hat{f}_{(\Gamma_1,...,\Gamma_d)}=0\in \Lambda_{V_d}^{S^2}[2d]$.
\end{lemma}

\begin{lemma} Take $(\Gamma_1,\dots, \Gamma_d)$ a cycle-free homogenous $d$-partition of $K_{2d}$, and  $1\leq x<y<z\leq 2d$. Then, there exist $(\Lambda_1,\dots, \Lambda_d)$, a unique cycle-free homogenous $d$-partition of $K_{2d}$  such that the two partitions $(\Gamma_1,\dots,\Gamma_d)$  and $(\Lambda_1,\dots,\Lambda_d)$ coincide on every edge of $K_{2d}$ except on the edges $(x,y)$, $(x,z)$, and $(y,z)$ where they are different on at least two edges. We will denote $(\Lambda_1,\dots,\Lambda_d)$ by $(\Gamma_1,\dots,\Gamma_d)^{(x,y,z)}$.
With the above notations we have $$\hat{f}_{(\Gamma_1,\dots,\Gamma_d)}=-\hat{f}_{(\Gamma_1,\dots,\Gamma_d)^{(x,y,z)}}\in {\Lambda}^{S^2}_{V_d}[2d].$$

\label{keylemma}
\end{lemma}

\begin{remark} Lemma \ref{generators}  and Lemma \ref{keylemma} were used in  \cite{lss} to show  existence of the map $det^{S^2}$ when $d=3$. In that case the map $det^{S^2}$ can be written as a sum over all cycle-free $d$-partitions of the complete graph $K_{2d}$. The approach in this paper is different, but the above two lemmas provide an intriguing connection between linear algebra  behind $\Lambda_{V_d}^{S^2}[2d]$, and the combinatorics of $\mathcal{P}_d(K_{2d})$. We will discuss this connection in the last section. 

Finally, one should notice that (when $d=2$ or $d=3$) the map $det^{S^2}$ is invariant under the action of the group $SL_d(k)$ on $V_d$, and so, by general results from invariant theory (\cite{stu}), it can be written as a sum of products of determinants of matrices with columns the vectors $v_{i,j}$ (see \cite{sv}). 
\end{remark}

\section{The Existence of the $det^{S^2}$ map}

Take $V_d$ a vector space of dimension $d$, and  fix $v_{i,j}\in V_d$ for all $1\leq i<j\leq 2d$. For each $1\leq k\leq 2d$ consider the  vector equation $\mathcal{E}_k((v_{i,j})_{1\leq i<j\leq 2d})$ defined as
\begin{eqnarray}
\sum_{s=1}^{k-1}(-1)^{s-1}\lambda_{s,k}v_{s,k}+ \sum_{t=k+1}^{2d}(-1)^{t}\lambda_{k,t}v_{k,t}=0.
\label{eqk}
\end{eqnarray}
We denote by $\mathcal{S}((v_{i,j})_{1\leq i<j\leq 2d})$  the system consisting of vector equations $\mathcal{E}_k((v_{i,j})_{1\leq i<j\leq 2d})$ for all $1\leq k\leq 2d$. 

Notice that the $2d$ vector equations are dependent. Indeed, each vector $v_{i,j}$ appears twice in our system, first time in equation $\mathcal{E}_i((v_{i,j})_{1\leq i<j\leq 2d})$ with coefficient $(-1)^j\lambda_{i,j}$, and second time in the equation $\mathcal{E}_j((v_{i,j})_{1\leq i<j\leq 2d})$ with coefficient $(-1)^{i-1}\lambda_{i,j}$. This means that 
\begin{eqnarray}
\sum_{k=1}^{2d}(-1)^{k-1}\mathcal{E}_k((v_{i,j})_{1\leq i<j\leq 2d})=0, \label{eq-dep}
\end{eqnarray}
and so, when studding this system it is enough to consider any $2d-1$ of the $2d$ vector equations. 

\begin{definition}  Let $v_{i,j}\in V_d$ for all $1\leq i<j\leq 2d$. We denote by $A((v_{i,j})_{1\leq i< j\leq 2d})$ the $d(2d)\times d(2d-1)$ matrix of the system $\mathcal{S}((v_{i,j})_{1\leq i<j\leq 2d})$. We denote by $A_k((v_{i,j})_{1\leq i< j\leq 2d})$ the $d(2d-1)\times d(2d-1)$ matrix corresponding to the system obtained after eliminating equation $\mathcal{E}_k((v_{i,j})_{1\leq i<j\leq 2d})$. Finally, we denote by $M_k((v_{i,j})_{1\leq i< j\leq 2d})$ the $d\times d(2d-1)$ matrix of the vector equation $\mathcal{E}_k((v_{i,j})_{1\leq i<j\leq 2d})$. 
\end{definition}

When the  vectors $v_{i,j}$ are clear from the context, in the interest of shortening the notation, we will suppress the vectors $v_{i,j}$ and write $A$, $A_k$ and $M_k$ respectively. 
Notice that matrix $A((v_{i,j})_{1\leq i<j\leq 2d})$ can be obtained by staking all the $M_k((v_{i,j})_{1\leq i<j\leq 2d})$'s on top of each other, i.e. 
\begin{equation}
    A((v_{i,j})_{1\leq i<j\leq 2d})=\begin{pmatrix}
M_1((v_{i,j})_{1\leq i<j\leq 2d})\\
M_2((v_{i,j})_{1\leq i<j\leq 2d}) \\
\vdots \\
M_{2d}((v_{i,j})_{1\leq i<j\leq 2d}) \\
\end{pmatrix}.\label{equation3}
\end{equation}
Similarly, we have 
\begin{equation}
 A_k=\begin{pmatrix}
M_1(v_{i,j})_{1\leq i<j\leq 2d}) \\
M_2((v_{i,j})_{1\leq i<j\leq 2d})\\
\vdots \\
M_{k-1}((v_{i,j})_{1\leq i<j\leq 2d})\\
M_{k+1}((v_{i,j})_{1\leq i<j\leq 2d})\\
\vdots \\
M_{2d-1}((v_{i,j})_{1\leq i<j\leq 2d})\\
M_{2d}((v_{i,j})_{1\leq i<j\leq 2d}) \\
\end{pmatrix}.
\end{equation}

\begin{remark} Using the matrices $M_k((v_{i,j})_{1\leq i< j\leq 2d})$, the dependence between the equations  $\mathcal{E}_k$ (i.e. equation \ref{eq-dep}) can be rewritten as  
\begin{eqnarray}
\sum_{k=1}^{2d}(-1)^{k-1}M_k((v_{i,j})_{1\leq i< j\leq 2d})=0\in M_{d\times d(2d-1)}. \label{eq-dep2}
\end{eqnarray}
\end{remark}

\begin{example} When $d=2$ and $v_{i,j}=\begin{pmatrix}
 \alpha_{i,j} \\
 \beta_{i,j}
\end{pmatrix}$ for all $1\leq i<j\leq 4$, the system $\mathcal{S}((v_{i,j})_{1\leq i<j\leq 4})$ becomes  
\begin{equation}
    \begin{pmatrix}
\alpha_{1,2} & -\alpha_{1,3} & 0 &\alpha_{1,4} & 0 &0\\
\beta_{1,2} & -\beta_{1,3} & 0&\beta_{1,4}& 0&0\\
\alpha_{1,2} & 0 & -\alpha_{2,3} & 0 & \alpha_{2,4} & 0\\
\beta_{1,2} & 0 & -\beta_{2,3} & 0 & \beta_{2,4} & 0\\
0 & \alpha_{1,3} & -\alpha_{2,3} & 0 & 0 &\alpha_{3,4}\\
0 & \beta_{1,3} & -\beta_{2,3} & 0 & 0 &\beta_{3,4}\\
0 & 0 & 0 & \alpha_{1,4} & -\alpha_{2,4} &\alpha_{3,4}\\
0 & 0 & 0 & \beta_{1,4} & -\beta_{2,4} &\beta_{3,4}\\
\end{pmatrix}\begin{pmatrix}
\lambda_{1,2}\\
\lambda_{1,3}\\
\lambda_{2,3}\\
\lambda_{1,4}\\
\lambda_{2,4}\\
\lambda_{3,4}
\end{pmatrix}=0,\label{equation1}
\end{equation}
with  $A((v_{i,j})_{1\leq i<j\leq 4})$ being the $8\times 6$ matrix of the system. When $k=2$ we have
\begin{equation}
    A_2((v_{i,j})_{1\leq i<j\leq 4})=\begin{pmatrix}
\alpha_{1,2} & -\alpha_{1,3} & 0 &\alpha_{1,4} & 0 &0\\
\beta_{1,2} & -\beta_{1,3} & 0&\beta_{1,4}& 0&0\\
0 & \alpha_{1,3} & -\alpha_{2,3} & 0 & 0 &\alpha_{3,4}\\
0 & \beta_{1,3} & -\beta_{2,3} & 0 & 0 &\beta_{3,4}\\
0 & 0 & 0 & \alpha_{1,4} & -\alpha_{2,4} &\alpha_{3,4}\\
0 & 0 & 0 & \beta_{1,4} & -\beta_{2,4} &\beta_{3,4}\\
\end{pmatrix},\label{equation2}
\end{equation}
\begin{equation}
    M_2((v_{i,j})_{1\leq i<j\leq 4})=\begin{pmatrix}
\alpha_{1,2} & 0 & -\alpha_{2,3} & 0 & \alpha_{2,4} & 0\\
\beta_{1,2} & 0 & -\beta_{2,3} & 0 & \beta_{2,4} & 0\\
\end{pmatrix}.\label{equation4}
\end{equation}
\label{example1}
\end{example}

\begin{remark}
Computing the determinant of the  matrix  $A_1((v_{i,j})_{1\leq i<j\leq 4})$ one recovers the formula for $det^{S^2}$ from \cite{sta2} (see also \cite{sv}). A different approach  to get the same result is to show that $det(A_1((v_{i,j})_{1\leq i<j\leq 4}))$ satisfies the universality property of the map $det^{S^2}$.

Indeed, first notice that 
$$det(A_1((v_{i,j})_{1\leq i< j\leq 4}))=det(A_k((v_{i,j})_{1\leq i< j\leq 4})),$$
for all $1\leq k\leq 4$. This essentially follows follows from Equation \ref{eq-dep2}. For example, in order to show that $det(A_1((v_{i,j})_{1\leq i< j\leq 4}))=det(A_2((v_{i,j})_{1\leq i< j\leq 4}))$ we use elementary transformations and the fact that $M_2=M_1+M_3-M_4$. More precisely, in the matrix $$A_2((v_{i,j})_{1\leq i< j\leq 4})=\begin{pmatrix}
M_1((v_{i,j})_{1\leq i< j\leq 4}) \\
M_3((v_{i,j})_{1\leq i< j\leq 4}) \\
M_4((v_{i,j})_{1\leq i< j\leq 4}) \\
\end{pmatrix},$$
 we add  rows $M_3-M_4$ to $M_1$ (i.e. $R_3-R_5$ to $R_1$, respectively $R_4-R_6$ to $R_2$)  to get  
\begin{eqnarray*}
det(A_1((v_{i,j})_{1\leq i< j\leq 4}))&=&
det(\begin{pmatrix}
M_1((v_{i,j})_{1\leq i< j\leq 4})+M_3((v_{i,j})_{1\leq i< j\leq 4})-M_4((v_{i,j})_{1\leq i< j\leq 4})\\
M_3((v_{i,j})_{1\leq i< j\leq 4}) \\
M_4((v_{i,j})_{1\leq i< j\leq 4}) \\
\end{pmatrix})\\
&=&det(\begin{pmatrix}
M_2((v_{i,j})_{1\leq i< j\leq 4}) \\
M_3((v_{i,j})_{1\leq i< j\leq 4}) \\
M_4((v_{i,j})_{1\leq i< j\leq 4}) \\
\end{pmatrix}),
\end{eqnarray*}
and so $det(A_1((v_{i,j})_{1\leq i< j\leq 4}))=det(A_2((v_{i,j})_{1\leq i< j\leq 4}))$.

Next, notice that if $v_{2,3}=v_{2,4}=v_{3,4}$ then  column  three, five and six of the matrix $A$ are linearly dependent because their sum is zero. And so, the corresponding  columns of matrix $A_1$ are linearly dependent,  which  means that $det(A_1)=0$. Similarly, if $v_{1,3}=v_{1,4}=v_{3,4}$ then $det(A_2)=0$, if $v_{1,2}=v_{1,4}=v_{2,4}$ then $det(A_3)=0$, and if $v_{1,2}=v_{1,3}=v_{2,3}$  then $det(A_4)=0$. 

Finally, one can check that if $e_{1,2}=e_{1,4}=e_{2,3}=e_1$ and $e_{1,3}=e_{2,4}=e_{3,4}=e_2$ then  $det(A_1((e_{i,j})_{1\leq i< j\leq 4}))=1$, which combined with the uniqueness of the map $det^{S^2}$ proved in \cite{sta2} for $d=2$, it shows  that $det(A_1((e_{i,j})_{1\leq i< j\leq 4}))=det^{S^2}((e_{i,j})_{1\leq i< j\leq 4})$. 
\label{remark1}
\end{remark}

\begin{remark} The system in Remark \ref{remark1} gives a geometrical interpretation for the condition $det^{S^2}((v_{i,j})_{1\leq i< j\leq 4})=0$ that is equivalent with the one in \cite{sv}. Notice however that the setting here is slightly different, as we use a different set of vector equations.  Later in the paper we will come back to this geometrical interpretation and discuss the general case. In particular, we will strengthen the result for $d=3$ from \cite{sv}.  
\end{remark}

Next we show the existence of a nontrivial map $det^{S^2}$ map for any $d$. 

\begin{theorem}
Let $d\geq 2$ and  $V_d$ be a $d$-dimensional vector space. Define $det^{S^2}:V_d^{d(2d-1)}\to k$ determined by 
$det^{S^2}((v_{i,j})_{1\leq i< j\leq 2d})=det(A_1((v_{i,j})_{1\leq i< j\leq 2d}))$. Then $det^{S^2}$ is a nontrivial  multilinear map with the property that $det^{S^2}((v_{i,j})_{1\leq i< j\leq 2d})=0$ if there exist $1\leq x<y<z\leq 2d$ such that $v_{x,y}=v_{x,z}=v_{y,z}$. 
\end{theorem}
\begin{proof}

First noticed that for all $1\leq t\leq 2d$ we have 
$$det(A_1((v_{i,j})_{1\leq i< j\leq 2d}))=det(A_t((v_{i,j})_{1\leq i< j\leq 2d})).$$
Indeed, from Equation \ref{eq-dep2} we know that 
$$(-1)^{t}M_t((v_{i,j})_{1\leq i< j\leq 2d})=\sum_{k=1,\; k\neq t}^{2d}(-1)^{k-1}M_k((v_{i,j})_{1\leq i< j\leq 2d}).$$
In the matrix 
$$A_t((v_{i,j})_{1\leq i< j\leq 2d})
=\begin{pmatrix}
M_1((v_{i,j})_{1\leq i< j\leq 2d}) \\
M_2((v_{i,j})_{1\leq i< j\leq 2d})\\
\vdots\\
M_{t-1}((v_{i,j})_{1\leq i< j\leq 2d})\\
M_{t+1}((v_{i,j})_{1\leq i< j\leq 2d})\\
\vdots \\
M_{2d-1}((v_{i,j})_{1\leq i< j\leq 2d})\\
M_{2d}((v_{i,j})_{1\leq i< j\leq 2d})\\
\end{pmatrix},$$
we add $\sum_{k=2,\; k\neq t}^{2d}(-1)^{k-1}M_k$ to $M_1$ and to get 
$$B_t((v_{i,j})_{1\leq i< j\leq 2d})=\begin{pmatrix}
(-1)^{t}M_t((v_{i,j})_{1\leq i< j\leq 2d}) \\
M_2((v_{i,j})_{1\leq i< j\leq 2d})\\
\vdots\\
M_{t-1}((v_{i,j})_{1\leq i< j\leq 2d})\\
M_{t+1}((v_{i,j})_{1\leq i< j\leq 2d})\\
\vdots \\
M_{2d-1}((v_{i,j})_{1\leq i< j\leq 2d})\\
M_{2d}((v_{i,j})_{1\leq i< j\leq 2d}) \\
\end{pmatrix}.$$
Using properties of determinants and elementary transformations, we get 
\begin{eqnarray*}
det(A_t)((v_{i,j})_{1\leq i< j\leq 2d})&=&det(B_t)((v_{i,j})_{1\leq i< j\leq 2d})\\
&=&(-1)^{td}det(\begin{pmatrix}
M_t((v_{i,j})_{1\leq i< j\leq 2d})\\
M_2((v_{i,j})_{1\leq i< j\leq 2d})\\
\vdots\\
M_{t-1}((v_{i,j})_{1\leq i< j\leq 2d})\\
M_{t+1}((v_{i,j})_{1\leq i< j\leq 2d})\\
\vdots \\
M_{2d-1}((v_{i,j})_{1\leq i< j\leq 2d})\\
M_{2d}((v_{i,j})_{1\leq i< j\leq 2d}) \\
\end{pmatrix})\\
&=&(-1)^{td}(-1)^{d^2(t-2)}det(\begin{pmatrix}
M_2((v_{i,j})_{1\leq i< j\leq 2d})\\
\vdots\\
M_{t-1}((v_{i,j})_{1\leq i< j\leq 2d})\\
M_t ((v_{i,j})_{1\leq i< j\leq 2d})\\
M_{t+1}((v_{i,j})_{1\leq i< j\leq 2d})\\
\vdots\\
M_{2d-1}((v_{i,j})_{1\leq i< j\leq 2d})\\
M_{2d}((v_{i,j})_{1\leq i< j\leq 2d}) \\
\end{pmatrix})\\
&=&(-1)^{t(d^2+d)-2d^2}det(A_1((v_{i,j})_{1\leq i< j\leq 2d}))
\\
&=&det(A_1((v_{i,j})_{1\leq i< j\leq 2d})).
\end{eqnarray*}

Next we will show if  there exist $1\leq x<y<z\leq 2d$ such that $v_{x,y}=v_{x,z}=v_{y,z}=w$ then $det(A_1((v_{i,j})_{1\leq i< j\leq 2d}))=0$. Chose $1\leq t\leq 2d$ such that $x\neq t\neq y$, $t\neq z$. We know from the above remarks that $det(A_1((v_{i,j})_{1\leq i< j\leq 2d}))=det(A_t((v_{i,j})_{1\leq i< j\leq 2d})).$
Notice that in the matrix $A_t((v_{i,j})_{1\leq i< j\leq 2d})$ the columns $(x,y)$, $(x,z)$ and $(y,z)$ are dependent .
Indeed, since $t$ is distinct from $x$, $y$ and $z$, when removing equation $\mathcal{E}_t$ we do not affect matrices $M_x$, $M_y$ and $M_z$. Which gives us the following matrix
\begin{eqnarray}
    A_t((v_{i,j})_{1\leq i< j\leq 2d})&=&\begin{pmatrix}
\dots & 0 & \dots&0 & \dots &0&\dots\\
\dots & (-1)^{y}v_{x,y} & \dots&(-1)^{z}v_{x,z}& \dots&0&\dots\\
\dots & 0 & \dots&0& \dots&0&\dots\\
\dots & (-1)^{x-1}v_{x,y} & \dots&0& \dots&(-1)^{z}v_{y,z}&\dots\\
\dots & 0 & \dots&0& \dots&0&\dots\\
\dots & 0 & \dots&(-1)^{x-1}v_{x,z}& \dots&(-1)^{y-1}v_{y,z}&\dots\\
\dots & 0 & \dots&0& \dots&0&\dots\\
\end{pmatrix}\\
&=&\begin{pmatrix}
\dots & 0 & \dots&0 & \dots &0&\dots\\
\dots & (-1)^{y}w & \dots&(-1)^{z}w& \dots&0&\dots\\
\dots & 0 & \dots&0& \dots&0&\dots\\
\dots & (-1)^{x-1}w & \dots&0& \dots&(-1)^{z}w&\dots\\
\dots & 0 & \dots&0& \dots&0&\dots\\
\dots & 0 & \dots&(-1)^{x-1}w& \dots&(-1)^{y-1}w&\dots\\
\dots & 0 & \dots&0& \dots&0&\dots\\
\end{pmatrix}
\end{eqnarray}
If we denote by $c_{x,y}$ the column corresponding to the pair $(x,y)$ then  
$$(-1)^{z}c_{x,y}-(-1)^yc_{x,z}+(-1)^xc_{y,z}=0,$$
 which proves that if $v_{x,y}=v_{x,z}=v_{y,z}$  for some $1\leq x<y<z\leq 2d$ then $det(A_t((v_{i,j})_{1\leq i< j\leq 2d}))=0$. 

Finally, we want to show that $det(A_1(E_d))\neq 0$ where $E_d=(e_{i,j})_{1\leq i<j\leq 2d}$ is defined by Equation \ref{defEd}. It is enough to show that the system $\mathcal{S}(E_d)$ has only the trivial solution. We will prove this by induction. When $d=2$ this was checked in Remark \ref{remark1}. 

Notice that if $(\lambda_{i,j})_{1\leq i<j\leq 2d}$ is a solution for $\mathcal{S}(E_d)$ then $\lambda_{i,2d-1}=0=\lambda_{i,2d}$ for all $1\leq i\leq 2d-2$, and $\lambda_{2d-1,2d}=0$. 

Indeed, if $i$ is odd and $i\leq 2d-2$ then $e_{i,2d-1}=e_d$ and so equation $\mathcal{E}_i(E_d)$ becomes

\begin{eqnarray}
\sum_{s=1}^{i-1}(-1)^{s-1}\lambda_{s,i}e_{s,i}+ \sum_{t=i+1,\; t\neq 2d-1}^{2d}(-1)^{t}\lambda_{i,t}e_{i,t}+(-1)^{2d-1}\lambda_{i,2d-1}e_d=0.
\end{eqnarray}
Using the definition of $E_d$, notice that the vectors $e_{s,i}$  and $e_{i,t}\in\{e_1,e_2,\dots,e_{d-1}\}$,  for all $1\leq s\leq i-1$, and all  $i+1\leq t\leq 2d$, $t\neq 2d-1$. Since $\{e_1,\dots, e_d\}$ is a basis, we get $\lambda_{i,2d-1}=0$ if $i$ is odd. 

Similarly, if $i$ is even and $i\leq 2d-2$ we have $e_{i,2d}=e_d$ and so equation $\mathcal{E}_i(E_d)$ becomes
\begin{eqnarray}
\sum_{s=1}^{i-1}(-1)^{s-1}\lambda_{s,i}e_{s,i}+ \sum_{t=i+1,\; t\neq 2d}^{2d}(-1)^{t}\lambda_{i,t}e_{i,t}+(-1)^{2d}\lambda_{i,2d}e_d=0.
\end{eqnarray}
Notice that the vectors $e_{s,i}$ and $e_{i,t}\in\{e_1,e_2,\dots,e_{d-1}\}$ for all $1\leq s\leq i-1$, and all  $i+1\leq t\leq 2d-1$. Since $\{e_1,\dots, e_d\}$ is a basis, we get $\lambda_{i,2d}=0$ if $i$ is even. 

Next, we use equation $\mathcal{E}_{2d-1}(E_d)$  to get
\begin{eqnarray}
\sum_{s=1}^{2d-2}(-1)^{s-1}\lambda_{s,2d-1}e_{s,2d-1}+(-1)^{2d}\lambda_{2d-1,2d}e_{2d-1,2d}=0.
\end{eqnarray}
Since $\lambda_{i,2d-1}=0$ if $i$ is odd, $e_{2j,2d-1}=e_j$ for all $1\leq j\leq d-1$, and $e_{2d-1,2d}=e_d$ we get
\begin{eqnarray}
\sum_{s=1}^{d-1}(-1)^{2s-2}\lambda_{2s,2d-1}e_s+(-1)^{2d}\lambda_{2d-1,2d}e_d=0,
\end{eqnarray}
which obviously means that $\lambda_{2j,2d-1}=0$ for all $1\leq j\leq d-1$, and $\lambda_{2d-1,2d}=0$. 

Similarly, using equation $\mathcal{E}_{2d}(E_d)$, and the fact that $\lambda_{2j-1,2d}=0$ for all $1\leq j\leq d$ one can show that $\lambda_{2j,2d}=0$ for all $1\leq j\leq d-1$. 

To summarize, if $(\lambda_{i,j})_{1\leq i<j\leq 2d}$ is a solution for $\mathcal{S}(E_d)$ then $\lambda_{i,2d-1}=0=\lambda_{i,2d}$ for all $1\leq i\leq 2d-2$ and $\lambda_{2d-1,2d}=0$. This means that  $(\lambda_{i,j})_{1\leq i<j\leq 2(d-1)}$ is a solution for $\mathcal{S}(E_{d-1})$, and so by induction we get $\lambda_{i,j}=0$ for all $1\leq i<j\leq 2d$. 
\end{proof}

\begin{corollary}
Let $d\geq 2$ and  $V_d$ be a $d$-dimensional vector space. Then there exists a nontrivial linear map $det^{S^2}:V_d^{\otimes d(2d-1)}\to k$ with the property that $det^{S^2}(\otimes_{1\leq i< j\leq 2d} (v_{i,j}))=0$ if there exist $1\leq x<y<z\leq 2d$ such that $v_{x,y}=v_{x,z}=v_{y,z}$. 
\end{corollary}

\begin{corollary}
Let $d\geq 2$ and  $V_d$ be a $d$-dimensional vector space. Then $dim_k(\Lambda_{V_d}^{S^2}[2d])\geq 1$. 
\end{corollary}

\begin{remark} By abuse of notation we will denote $det^{S^2}$ both, the multilinear map on $V_d^{d(2d-1)}$, and the linear map on $V_d^{\otimes d(2d-1)}$.  This should not create any confusion. 
\end{remark}

\section{Applications}
\subsection{Geometrical Application}
We have the following geometrical interpretation for the condition $det^{S^2}=0$ that generalizes the results from \cite{sv}. 
\begin{proposition} Let $d\geq 2$, $V_d$ be a $d$-dimensional vector space, and take $(v_{i,j})_{1\leq i<j\leq 2d}\in V_d^{d(2d-1)}$. The following are equivalent:\\
(1) $det^{S^2}((v_{i,j})_{1\leq i<j\leq 2d})=0$. \\
(2) There exists $\lambda_{i,j}\in k$ for all $1\leq i<j\leq 2d$ not all zero such that 
\begin{eqnarray}
\sum_{s=1}^{k-1}(-1)^{s-1}\lambda_{s,k}v_{s,k}+ \sum_{t=k+1}^{2d}(-1)^{t}\lambda_{k,t}v_{k,t}=0.
\end{eqnarray}
for all $1\leq k\leq 2d$. 
\label{cor1}
\end{proposition}
\begin{proof} It follows from the definition of map $det^{S^2}$. 
\end{proof}

In \cite{sv} it was proved that if $p_i\in V_3$ for all $1\leq i\leq 6$, and we define $v_{i,j}=p_j-p_i$ then $det^{S^2}((v_{i,j})_{1\leq i<j\leq 6})=0$. It was conjectured that a similar result is true in general. Indeed  we have the following. 
\begin{corollary} Let $d\geq 2$ and  $V_d$ be a $d$-dimensional vector space. Take $p_i\in V_d$ for $1\leq i\leq 2d$ and define $v_{i,j}=p_j-p_i$ for all $1\leq i<j\leq 2d$. Then $det^{S^2}((v_{i,j})_{1\leq i<j\leq 2d})=0$. 
\end{corollary}
\begin{proof} 
We will show that  the system $\mathcal{S}((v_{i,j})_{1\leq i<j\leq 2d})$ has a nontrivial solution and so by Proposition  \ref{cor1}, we get that  $det^{S^2}((v_{i,j})_{1\leq i<j\leq 2d})=0$.

Consider the vectors $v_{1,j}=p_j-p_1$ for $2\leq j\leq 2d$. Since $dim_k(V_d)=d<2d-1$ then there exist $\lambda_{j}$ for $2\leq j\leq 2d$ not all zero such that 
\begin{eqnarray}
\sum_{t=2}^{2d}(-1)^{t}\lambda_{t}v_{1,t}=0.\label{defL1}
\end{eqnarray}
Case I. Assume that $\displaystyle{\Lambda=\sum_{t=2}^{2d}(-1)^{t}\lambda_{t}\neq 0}$. For $1\leq i< j\leq 2d$ take
\begin{equation}
\lambda_{i,j}=\begin{cases}
		\lambda_j ~~{\rm if}~~i=1,\\
\\
		\frac{\lambda_{i}\lambda_{j}}{\Lambda} ~~{\rm if}~~i> 1.\\
\end{cases}
\label{lij1}
\end{equation}
We will check that  $\{\lambda_{i,j}\}_{1\leq i<j\leq 2d}$ is a nontrivial solution for the system $\mathcal{S}((v_{i,j})_{1\leq i<j\leq 2d})$.

Equation $\mathcal{E}_1$ is satisfied because of the definition of $\lambda_{1,j}$ (see Equation \ref{defL1}). 
Take $2\leq k\leq 2d$, notice that $v_{s,k}=p_k-p_s=v_{1,k}-v_{1,s}$ if $1< s\leq k-1$, and  $v_{k,t}=p_t-p_k=v_{1,t}-v_{1,k}$ if $k+1\leq t\leq 2d$. So we have:
\begin{allowdisplaybreaks}
\begin{eqnarray*}
&&\sum_{s=1}^{k-1}(-1)^{s-1}\lambda_{s,k}v_{s,k}+\sum_{t=k+1}^{2d}(-1)^{t}\lambda_{k,t}v_{k,t}\\
&=&\lambda_{k}v_{1,k}+\sum_{s=2}^{k-1}(-1)^{s-1}\frac{\lambda_{s}\lambda_{k}}{\Lambda}(v_{1,k}-v_{1,s})+\sum_{t=k+1}^{2d}(-1)^{t}\frac{\lambda_{k}\lambda_{t}}{\Lambda}(v_{1,t}-v_{1,k})\\
&=&(\lambda_{k}+\sum_{s=2}^{k-1}(-1)^{s-1}\frac{\lambda_{k}\lambda_{s}}{\Lambda}+\sum_{t=k+1}^{2d}(-1)^{t-1}\frac{\lambda_{t}\lambda_{k}}{\Lambda})v_{1,k}+
\sum_{s=2}^{k-1}(-1)^{s}\frac{\lambda_{s}\lambda_{k}}{\Lambda}v_{1,s}+\sum_{t=k+1}^{2d}(-1)^{t}\frac{\lambda_{k}\lambda_{t}}{\Lambda}v_{1,t}\\
&=&(\lambda_{k}+\sum_{s=2}^{k-1}(-1)^{s-1}\frac{\lambda_{k}\lambda_{s}}{\Lambda}+(-1)^{k-1}\frac{\lambda_{k}^2}{\Lambda}v_{1,k}+\sum_{t=k+1}^{2d}(-1)^{t-1}\frac{\lambda_{t}\lambda_{k}}{\Lambda})v_{1,k}\\
&&+\sum_{s=2}^{k-1}(-1)^{s}\frac{\lambda_{s}\lambda_{k}}{\Lambda}v_{1,s}+(-1)^{k}\frac{\lambda_{k}^2}{\Lambda}v_{1,k}+\sum_{t=k+1}^{2d}(-1)^{t}\frac{\lambda_{k}\lambda_{t}}{\Lambda}v_{1,t}\\
&=&(\lambda_{k}-\lambda_{k}\frac{\Lambda}{\Lambda})v_{1,k}+\frac{\lambda_{k}}{\Lambda}\sum_{s=2}^{2d}(-1)^{s}\lambda_{s}v_{1,s}\\
&=&0,
\end{eqnarray*}
\end{allowdisplaybreaks}
and so equation $\mathcal{E}_k$ is satisfied.

Case II. Assume that $\displaystyle{\Lambda=\sum_{t=2}^{2d}(-1)^{t}\lambda_{t}=0}$. Take $2\leq a\leq 2d-1$  with the property that $\lambda_a\neq 0$ and $\lambda_i=0$ for all $2\leq i<a$ (such an integer exists since not all $\lambda_j$ are zero). 
For all $1\leq i<j\leq 2d$ we define 

\begin{equation}
\lambda_{i,j}=\begin{cases}
		0 ~~{\rm if}~~i<a,\\
\\
		\frac{\lambda_{i}\lambda_{j}}{\lambda_a} ~~{\rm if}~~i\geq a.\\
\end{cases}
\label{lij2}
\end{equation}

We will check that $\{\lambda_{i,j}\}_{1\leq i<j\leq 2d}$ is a solution for the system $\mathcal{S}((v_{i,j})_{1\leq i<j\leq 2d})$. 

If $k<a$ then the equation $\mathcal{E}_k$ is trivially satisfied since when $k<a$ we have  $\lambda_{i,k}=0$ for all $1 \leq i<k$, and $\lambda_{k,j}=0$  for all $k<j\leq 2d$.  

Next assume that $k=a$. We use the fact that $v_{i,j}=p_j-p_i=v_{1,j}-v_{1,i}$ for $1\leq i<j\leq 2d$, $\lambda_{s,a}=0$ for all $1\leq s<a-1$, and $\lambda_{a,t}=\frac{\lambda_t\lambda_a}{\lambda_a}$ for all $a<t\leq 2d$ to get

\begin{allowdisplaybreaks}
\begin{eqnarray*}
&&\sum_{s=1}^{a-1}(-1)^{s-1}\lambda_{s,a}v_{s,a}+\sum_{t=a+1}^{2d}(-1)^{t}\lambda_{a,t}v_{a,t}\\
&=&\sum_{t=a+1}^{2d}(-1)^{t}\frac{\lambda_{t}\lambda_{a}}{\lambda_{a}}(v_{1,t}-v_{1,a})\\
&=&\sum_{t=a+1}^{2d}(-1)^{t}\lambda_{t}v_{1,t} -(\sum_{t=a+1}^{2d}(-1)^{t}\lambda_{t})v_{1,a}\\
&=&(-1)^{a-1}\lambda_av_{1,a}-(-1)^{a-1}\lambda_av_{1,a}\\
&=&0.
\end{eqnarray*}
\end{allowdisplaybreaks}
Where the next to last equality follows because $\displaystyle{\sum_{t=a}^{2d}(-1)^t\lambda_tv_{1,t}=0}$, and $\displaystyle{\sum_{t=a}^{2d}(-1)^t\lambda_t=0}$.

Finally, assume that $k>a$. We use the fact that $v_{i,j}=p_j-p_i=v_{1,j}-v_{1,i}$ for $2\leq i<j\leq 2d$,  that $\lambda_{i,k}=0$ for all $1\leq i<a-1$, $\lambda_{s,k}=\frac{\lambda_k\lambda_s}{\lambda_a}$ for all $a\leq s<k-1$, and $\lambda_{k,t}=\frac{\lambda_t\lambda_k}{\lambda_a}$ for all $k<t\leq 2d$ to get

\begin{allowdisplaybreaks}
\begin{eqnarray*}
&&\sum_{s=1}^{k-1}(-1)^{s-1}\lambda_{s,k}v_{s,k}+\sum_{t=k+1}^{2d}(-1)^{t}\lambda_{k,t}v_{k,t}\\
&=&\sum_{s=a}^{k-1}(-1)^{s-1}\frac{\lambda_{s}\lambda_{k}}{\lambda_a}(v_{1,k}-v_{1,s})+\sum_{t=k+1}^{2d}(-1)^{t}\frac{\lambda_{k}\lambda_{t}}{\lambda_a}(v_{1,t}-v_{1,k})\\
&=&\sum_{s=a}^{k-1}(-1)^{s-1}\frac{\lambda_{s}\lambda_{k}}{\lambda_a}v_{1,k}+(-1)^{k-1}\frac{\lambda_{k}\lambda_{k}}{\lambda_a}v_{1,k}+\sum_{t=k+1}^{2d}(-1)^{t-1}\frac{\lambda_{k}\lambda_{t}}{\lambda_a}v_{1,k}+\\
&&\sum_{s=a}^{k-1}(-1)^{s}\frac{\lambda_{s}\lambda_{k}}{\lambda_a}v_{1,s}+(-1)^{k}\frac{\lambda_{k}\lambda_{k}}{\lambda_a}v_{1,k}+\sum_{t=k+1}^{2d}(-1)^{t}\frac{\lambda_{k}\lambda_{t}}{\lambda_a}v_{1,t}\\
&=&\frac{\lambda_{k}}{\lambda_a}(\sum_{s=a}^{2d}(-1)^{s-1}\lambda_{s})v_{1,k}+\frac{\lambda_{k}}{\lambda_a}\sum_{t=a}^{2d}(-1)^{t}\lambda_tv_{1,t}\\
&=&0,
\end{eqnarray*}
\end{allowdisplaybreaks}
and so equation $\mathcal{E}_k$ is satisfied. 

This means that the  system $\mathcal{S}((v_{i,j})_{1\leq i<j\leq 2d})$ has a nontrivial solution, and so by Proposition \ref{cor1} we get that  $det^{S^2}((v_{i,j})_{1\leq i<j\leq 2d})=0$.

\end{proof}

\subsection{Cycle-free $d$-partitions of the complete graph $K_{2d}$}

Next we give an application of the map $det^{S^2}$ to combinatorics. More precisely we show that the $det^{S^2}$ map can detect if a $d$-partition of the complete graph $K_{2d}$ is cycle-free or not. 
 
\begin{theorem} Let $(\Gamma_1,\dots,\Gamma_d)\in \mathcal{P}_d(K_{2d})$ a $d$-partition of the complete graph $K_{2d}$. Then $(\Gamma_1,\dots,\Gamma_d)$ is cycle-free if and only if $det^{S^2}(f_{(\Gamma_1,\dots,\Gamma_d)})\neq 0$. 
\end{theorem}
\begin{proof}
First we show that if $(\Gamma_1,\dots,\Gamma_d)$ is a  cycle-free $d$-partition of $K_{2d}$ then $det^{S^2}(f_{(\Gamma_1,\dots,\Gamma_d)})\neq 0$. Indeed, from the definition of $det^{S^2}$, we know that if $det^{S^2}(f_{(\Gamma_1,\dots,\Gamma_d)})= 0$ then  the system $\mathcal{S}(f_{(\Gamma_1,\dots,\Gamma_d)})$ has a nontrivial solution $(\mu_{i,j})_{1\leq i<j\leq 2d}$.  Take $G$ the sub graph of $K_{2d}$ determined by  $V(G)=V(K_{2d})$, and $E(G)=\{(a,b)\in E(K_{2d})|\mu_{a,b}\neq 0\}$. Since the solution $(\mu_{i,j})_{1\leq i<j\leq 2d}$ is nontrivial then $E(G)$ is not empty. The partition $(\Gamma_1,\dots,\Gamma_d)$ is inducing a partition $(\Phi_1,\dots,\Phi_{d})$ of $G$. Take $x\in \{1,\dots,d\}$ such that $E(\Phi_x)\neq \emptyset$ (such an $x$ exists since $E(G)\neq \emptyset$).  
Because $\Gamma_x$ is cycle-free it follows that $\Phi_x$ is also cycle free, and so we can find a vertex $a\in \{1,\dots,2d\}$ such that the degree of the vertex $a$ in the graph $\Phi_x$ is one (i.e. there exists exactly one other vertex $b$ with the property that the edge $(a, b)\in E(\Phi_x)$). 

For simplicity we will assume that $a<b$ (the case $a>b$ is similar). Consider the equation $\mathcal{E}_a((w_{i,j})_{1\leq i<j\leq 2d})$ where  $(w_{i,j})_{1\leq i<j\leq 2d}=f_{(\Gamma_1,\dots,\Gamma_d)}$. We have
\begin{eqnarray}
\sum_{s=1}^{a-1}(-1)^{s-1}\mu_{s,a}w_{s,a}+ \sum_{t=a+1}^{2d}(-1)^{t}\mu_{a,t}w_{a,t}=0.
\label{eqk2}
\end{eqnarray}
Notice that    $w_{s,a}$, $w_{a,t}\in \{e_1,\dots e_d\}$. Moreover $w_{s,a}\neq e_x$ for all $1\leq s\leq a-1$, and for $a+1\leq t\leq 2d$ we have $w_{a,t}=e_x$ if and only if $t=b$. Finally, since the coefficient of $w_{a,b}=e_x$ is $\mu_{a,b}\neq 0$, we get that equation $\mathcal{E}_a$ gives a nontrivial linear dependence relation among the vectors $\{e_1,\dots,e_d\}$, which is obviously a contradiction since $\mathcal{B}_d$ is a basis.   This means that the system  $\mathcal{S}(f_{(\Gamma_1,\dots,\Gamma_d)})$ has only the trivial solution, and so $det^{S^2}(f_{(\Gamma_1,\dots,\Gamma_d)})\neq 0$. 

For the converse we prove a more general statement. Take $D:V_d^{\otimes d(2d-1)}\to k$ to be a linear map with the property that $D(\otimes_{1\leq i<j\leq 2d} (v_{i,j}))=0$ if there exist $1\leq x<y<z\leq 2d$ such $v_{x,y}=v_{x,z}=v_{y,z}$. It follows from results proved in  \cite{lss}  (see Lemma \ref{generators}), that if the $d$-partition $(\Gamma_1,\dots,\Gamma_d)$ is not cycle-free, then  that $D(f_{(\Gamma_1,\dots,\Gamma_d)})=0$. However, this exact statement was not made explicitly in \cite{lss}, so for completeness sake, we recall the main steps of that proof.  

Let $(\Gamma_1,\dots,\Gamma_d)$ be a $d$-partition of $K_{2d}$ and  $f_{(\Gamma_1,\dots,\Gamma_d)}=\otimes_{1\leq i<j\leq 2d} (f_{i,j})\in \mathcal{G}_{\mathcal{B}_d}[2d]$.
If $(\Gamma_1,\dots,\Gamma_d)$ is not cycle free then there exists $1\leq k\leq d$ such that $\Gamma_k$ has a cycle. If $\Gamma_k$ has cycle of length $3$, then there exist $1\leq x<y<z\leq 2d$ such that $(x,y)$, $(x,z)$ and $(y,z)\in E(\Gamma_k)$, and so $f_{x,y}=f_{x,z}=f_{y,z}=e_k$. By the property of the map $D$ this  means that $D(f_{(\Gamma_1,\dots,\Gamma_d)})=0$. 

If the length of the cycle in $\Gamma_k$ is  $l>4$, then it was shown in \cite{lss} (see Lemma 3.5), that there exist two $d$-partition $(\Gamma_1^{(1)},\dots,\Gamma_d^{(1)})$ and $(\Gamma_1^{(2)},\dots,\Gamma_d^{(2)})$ of $K_{2d}$ such that $\Gamma_{k}^{(1)}$ and $\Gamma_{k}^{(2)}$ have cycles of length $l-1$, and 
$$D(f_{(\Gamma_1,\dots,\Gamma_d)})+D(f_{(\Gamma_1^{(1)},\dots,\Gamma_d^{(1)})})+D(f_{(\Gamma_1^{(2)},\dots,\Gamma_d^{(2)})})=0.$$
From here, the statement follows by induction. 
\end{proof}

\begin{remark}
The result in this section is consistent with conjecture from \cite{lss} where it was proposed that the $det^{S^2}$ map can be written as a sum over  the set of homogeneous cycle free $d$-partition of the graph $K_{2d}$ 
\begin{equation}
det^{S^2}(\otimes_{1\leq i<j\leq 2d} (v_{i,j}))=\sum_{(\Gamma_1,...,\Gamma_d)\in \mathcal{P}^{h,cf}_d(K_{2d})} \varepsilon_d^{S^2}((\Gamma_1,...,\Gamma_d))M_{(\Gamma_1,...,\Gamma_d)}((v_{i,j})_{1\leq i<j\leq 2d}),
\label{detS2d}
\end{equation}
where $\varepsilon_d^{S^2}:\mathcal{P}^{h,cf}_d(K_{2d})\to \{-1,1\}$, and $M_{(\Gamma_1,...,\Gamma_d)}((v_{i,j})_{1\leq i<j\leq 2d})$ is a monomial associated to $(\Gamma_1,...,\Gamma_d)$ and to the element $\otimes_{1\leq i<j\leq 2d}(v_{i,j})\in V_d^{\otimes d(2d-1)}$ (see \cite{lss} for more details).  In order to prove this formula for every $d$, it would be enough to show that the involutions $(\Gamma_1,\dots,\Gamma_d)\to (\Gamma_1,\dots,\Gamma_d)^{(x,y,z)}$ (described in Lemma \ref{keylemma}) act transitively on $\mathcal{P}_d(K_{2d})$. This would also imply the uniqueness up to a constant of the map $det^{S^2}$. 
\end{remark}

\begin{remark} The construction of $det^{S^2}$ is somehow similar with the definition the resultant of two polynomials (see \cite{gkz}). It might be interesting to understand if there is a more general construction that cover both examples. 
\end{remark}
\begin{remark}
$\Lambda_{V_d}^{S^3}$ was introduced in \cite{ls2} as another generalization of the exterior algebra. When $d=2$, it was shown that there exists a linear map $det^{S^3}:V_2^{\otimes 20}\to k$ with the property that $det^{S^3}(\otimes_{1\leq i<j<k\leq 6}(v_{i,j,k}))=0$ if there exist $1<x<y<z<t\leq 6$ such that $v_{x,y,z}=v_{x,y,t}=v_{x,z,t}=v_{y,z,t}$. 
We expect that the construction presented in this paper can be adapted to define a map $det^{S^3}$ for any $d$.  
\end{remark}


\bibliographystyle{amsalpha}

\end{document}